\documentclass[11pt,letterpaper]{amsart}
\usepackage{latexsym,amsfonts,amssymb,amsmath,amsthm, datetime2}
\usepackage{graphicx}
\usepackage{color}
\usepackage{mathrsfs}
\usepackage{amsmath, amsthm, amsfonts, amssymb}
\usepackage{wrapfig}
\usepackage{stackrel}
\usepackage{color}
\usepackage{float}
\usepackage{enumitem}
\usepackage{mathtools}
\usepackage{multicol}
\usepackage[normalem]{ulem}
\usepackage{xcolor}
\usepackage{hyperref}
\usepackage{cancel}

\usepackage[margin=1in]{geometry}

\newcommand{\half}{\ensuremath{ \frac{1}{2}}}

\newcommand{\ord}{\text{ord}}

\newcommand{\thalf}{\tfrac12}

\newtheorem{lem}{Lemma}
\newtheorem{prop}{Proposition}

\theoremstyle{plain}		
	\newtheorem{mytheo}{Theorem}[section]

\theoremstyle{remark}

\numberwithin{equation}{section}

\begin{document}
\author{Rizwanur Khan}

 \address{Department of Mathematical Sciences\\ University of Texas at Dallas\\ Richardson, TX 75080-3021}
 \email{rizwanur.khan@utdallas.edu}

  \thanks{The author was supported by the National Science Foundation grant DMS-2341239.}

\title[First Moment of Dirichlet $L$-functions over Character Group Generators]{The First Moment of Dirichlet $L$-functions over Generators of the Character Group}

\begin{abstract}
We evaluate the average of the central values $L(1/2,\chi)$ over the generators of the group of Dirichlet characters modulo $q$, as $q\to\infty$ through the primes.
\end{abstract}

\maketitle

\section{Introduction}

The study of ranks of elliptic curves is a central problem in modern number theory, closely intertwined with the theory of $L$-functions. An interesting question in this direction is to obtain an upper bound for the analytic rank of an elliptic curve $E/\mathbb{Q}$ over the cyclotomic extension $\mathbb{Q}(e^{2\pi i/q})$, in terms of $q$. This leads naturally to the study of $L$-functions in special families of Dirichlet characters.

Rohrlich \cite{Rohrlich} proved that the analytic rank is bounded for $q$ any large power of a fixed prime. A key step in his method was an asymptotic evaluation of the mean of central $L$-values
\begin{align*}
 \sum_{\substack{ \chi \bmod q \\ \text{ord}(\chi)=r}} L(\thalf, E\otimes \chi),
\end{align*}
where the summation is over the family of primitive Dirichlet characters of modulus $q$ and order $r$, for permissible values $r|\phi(q)$. This family can also be interpreted as an orbit of the natural action of Gal$(\mathbb{\overline{Q}}/\mathbb{Q})$ on the primitive characters. For prime values of $q$, Chinta \cite{Chinta} showed that the analytic rank can grow with $q$ at worst like a small power of $q$. This exponent was later improved by Dasgupta and Khan \cite{DasguptaKhan}. In this case, it seems much harder to evaluate the above mean. Chinta's idea to get around this difficulty was to instead study the sum
\begin{align*}
 \sum_{\substack{ \chi \bmod q \\ \text{ord}(\chi)=r}} L(\thalf, E\otimes \chi) M(\thalf, \chi),
\end{align*} 
where $M(\thalf, \chi)$ is an analytic device called a mollifier. The mollifier is designed to approximate the inverse of  $L(\thalf, E\otimes \chi)$, so that $L(\thalf, E\otimes \chi) M(\thalf, \chi)$ is expected to be close to $1$. This simplifies the evaluation of the sum above, at least when $r$ is sufficiently large relative to $q$. However, this approach leaves open the problem of understanding the mean of the $L$-functions themselves, among characters of a given order. Motivated by the importance of these families, we take a first step towards this goal by considering a more basic problem. For $q$ prime, we study the mean of the central values $L(\half, \chi)$ of Dirichlet $L$-functions over the characters mod $q$ of order $q-1$, which are precisely the generators of the character group. This is the simplest $L$-function associated to $\chi$ and the largest family of a fixed order, with $\phi(q-1)$ elements, yet surprisingly the problem is already nontrivial. 
\begin{mytheo} \label{main} For $q$ a prime number, we have
\begin{align*}
\frac{1}{\phi(q-1)} \sum_{\substack{ \chi \bmod q \\ \mathrm{ord}(\chi)=q-1 }} L(\thalf,  \chi) = 1 + O((\log q)^{-\delta})
\end{align*}
for some $\delta>0$.
\end{mytheo}
\noindent The proof gives $\delta=\frac{3}{10}$, although we have not tried to optimize it.

As we will see, after using an approximate functional equation and averaging over characters, the problem boils down to showing that
\[
\sum_{\substack{2\le n \le q^{1+\varepsilon}\\ (n,q)=1}} \, \frac{\mu(\ord(n))}{n^\half \phi(\ord(n))} \ll (\log q)^{-\delta},
\]
where $\ord(n)$ is the multiplicative order of $n$ modulo $q$. The term corresponding to $n=1$ has been excluded here because it forms the main term of Theorem \ref{main}. It seems difficult to track the sign changes of $\mu(\ord(n))$, so we apply absolute values term-wise, thereby forsaking any potential cancelation arising from these factors. Proving the required estimate for the resulting sum of positive terms turns out to be quite delicate, but nevertheless elementary methods suffice. Working over dyadic intervals, we need to obtain some control over the number of integers in $[N,2N]$ of order $k$, where $k \mid \phi(q)$. This is reminiscent of  the equidistribution results of Bourgain, Glibichuk, and Konyagin \cite{BGK} for cyclic subgroups of $(\mathbb{Z}/q\mathbb{Z})^\times$ of size $k$, although their results apply only on a larger scale, namely when $k$ exceeds a power of $q$, than is relevant here. 

Note that if $q=2^p-1$ is a Mersenne prime, then 
\[
\sum_{\substack{2\le n \le q^{1+\varepsilon}\\ (n,q)=1}} \, \frac{|\mu(\ord(n))|}{n^\half \phi(\ord(n))}  \ge \frac{1}{2^\half \phi(\ord(2))} = \frac{1}{2^\half \phi(p)} \gg (\log q)^{-1}.
\]
This shows that one cannot expect to do better than a logarithmic saving in Theorem \ref{main} unless one finds a way to utilize cancelation from $\mu(\ord(n))$.

We remark that in the case when the modulus is a large power of a fixed prime, the first and second moments of $L(\thalf,  \chi)$ in families of primitive Dirichlet characters of any given order, are known by the work of Khan, Milicevi\'{c}, and Ngo \cite{KMN}. The second moment in this case requires the deep input of the Roth-Ridout theorem for $p$-adic diophantine approximation.

 \section{Preliminaries}
 
Throughout, $q$ denotes an odd prime. 

We have (see \cite[Theorem 328]{hardy-wright}) that
\begin{align}
\label{phi-lower} \phi(n) \gg \frac{n}{\log\log n}.
\end{align}

Define, where $p$ denotes a prime, the divisor functions
\begin{align}
\label{divisor-def} d(n;x)=\sum_{\substack{d|n\\ d \le x}}1 \ \ \ \ \ \ \ \ \  \ \text{and} \ \ \ \ \ \ \ \ \  \ d_P(n;x)=\sum_{\substack{d|n\\ d\le x\\ p|d \text{ for some } p\ge P}}1.
\end{align}

Define the smooth number counting function (where $p$ denotes a prime)
\begin{align}
\label{smooth-def} \Psi(x;y) := \sum_{\substack{n\le x \\ p|n \Rightarrow p\le y}}1.
\end{align}
We will use the following basic bound in the theory of smooth numbers.
\begin{lem}\label{smooth} \cite[\S III.5.1 Theorem 1]{ten}
For $x\ge y\ge 2$, we have 
\[
\Psi(x;y)  \ll x \exp\Big(-\frac{\log x}{2\log y}\Big). 
\]
\end{lem}

To average character values over our family, we will need
 \begin{lem} \cite[Proposition 1(i)]{Chinta} \label{averaging}
For $d|(q-1)$ and $(n,q)=1$, we have
\[
 \sum_{\substack{ \chi \bmod q \\ \mathrm{ord}(\chi)=\frac{q-1}{d} }} \chi(n) = \phi\Big(\frac{q-1}{d}\Big) \frac{\mu(\mathrm{ord}(n^d))}{\phi(\mathrm{ord}(n^d))}.
\]
If $(n,q)=q$, then the sum on the left is $0$.
 \end{lem}

We will need an expression for $L(\half, \chi)$ as a Dirichlet sum.
 \begin{lem} \label{afe}
Let $\chi$ be a Dirichlet character mod $q$ of order $q-1$. Then
\[
L(\thalf, \chi) = \sum_{n\le q^\frac{51}{50}} \frac{\chi(n)}{n^\half} W(n) + O(q^{-10})
\]
for some function $W$, depending on $q$, with $W(n)\ll 1$ and $W(1)=1+O(q^{-1})$.
 \end{lem}
\proof
We first observe that $\chi$ must be an odd character (i.e. $\chi(-1)=-1$). This is because $\chi$ has maximal order, so powers of $\chi$ will give rise to the entire group of Dirichlet characters mod $q$. If $\chi$ were even (i.e. $\chi(-1)=1$) then all characters would be even, which is not true. Thus $L(s,\chi)$ satisfies the functional equation for an odd character,
\[
q^{\frac{s}{2}} \pi^{-\frac{s}{2}} \Gamma\Big(\frac{1+s}{2}\Big) L(s,\chi) = \frac{\tau(\chi)}{iq^\half} q^{\frac{1-s}{2}} \pi^{-\frac{1-s}{2}} \Gamma\Big(\frac{2-s}{2}\Big) L(1-s,\overline{\chi}),
\]
where $\tau(\chi)$ is the Gauss sum (and recall that $|\tau(\chi)|=q^\half$). Then by a standard approximate functional equation, given by \cite[Theorem 5.3]{iwakow} with $X=q^\frac{51}{100}$, we have
\[
L(\thalf, \chi) = \sum_{n \ge 1} \frac{\chi(n)}{n^\half} V\Big( n q^{-\frac{101}{100}} \Big) +   \frac{\tau(\chi)}{iq^\half}  \sum_{n\ge 1} \frac{\overline{\chi}(n)}{n^\half} V\Big(nq^\frac{1}{100}\Big),
\]
where for $x>0$, 
\[
V(x):=\frac{1}{2\pi i } \int_{(2)} (\pi^\half x)^{-s} \Gamma(\tfrac34+\tfrac{s}{2}) \Gamma(\tfrac{3}{4})^{-1} \frac{ds}{s}.
\]
For $0<x\le 1 $, we move the line of integration left to $\Re(s)= -1$, crossing a simple pole at $s=0$, to get $V(x)= 1+ O(x).$  For $x>1$, we move the line of integration right to $\Re(s)= 10000$, where we get the bound $V(x)\ll x^{-10000}$. Either way, we have $V(x)\ll 1$ and we put $W(n) := V(nq^{-\frac{101}{100}})$ for $n\le q^\frac{51}{50}$, so that
\[
L(\thalf, \chi) = \sum_{n\le q^\frac{51}{50}} \frac{\chi(n)}{n^\half} W(n)  + \sum_{n > q^\frac{51}{50} } \frac{\chi(n)}{n^\half} V\Big( n q^{-\frac{101}{100}} \Big)+ \frac{\tau(\chi)}{iq^\half}  \sum_{n\ge 1} \frac{\overline{\chi}(n)}{n^\half} V\Big(nq^\frac{1}{100}\Big).
\]
In the second sum, where $n > q^\frac{51}{50}$, we have $n q^{-\frac{101}{100}} \ge n^\frac{1}{10000} q^\frac{49}{5000}$. Using $V(x)\ll x^{-10000}$, we get that
 $V(n q^{-\frac{101}{100}})\ll n^{-1}q^{-10}$. The same bound holds for $V(nq^\frac{1}{100})$. Then trivially bounding we see that the last two sums are $O(q^{-10})$. 
\endproof


 \section{Initial calculations}\label{sec}

By Lemmas \ref{afe} and \ref{averaging}, with $d=1$, we have
 \begin{align}
\frac{1}{\phi(q-1)}  \sum_{\substack{ \chi \bmod q \\ \mathrm{ord}(\chi)=q-1 }}  L(\thalf,  \chi)  &=  \sum_{\substack{ n\le q^{\frac{51}{50}}\\ (n,q)=1}}  \frac{1}{n^\half}  \frac{\mu(\mathrm{ord}(n))}{\phi(\mathrm{ord}(n))}W(n) + O(q^{-10})\\
\nonumber &= 1+  \sum_{\substack{2\le n\le q^{\frac{51}{50}}\\ (n,q)=1}}  \frac{1}{n^\half}   \frac{\mu(\mathrm{ord}(n))}{\phi(\mathrm{ord}(n))}W(n) + O(q^{-1}).
 \end{align}
Applying absolute values term-wise, we see that to establish Theorem \ref{main}, it suffices to prove
 \[
  \sum_{\substack{2\le n\le q^{\frac{51}{50}}\\ (n,q)=1}}   \frac{1}{n^\half \phi(\mathrm{ord}(n))} \ll (\log q)^{-\delta}.
 \]
We can easily discard the terms with $n\ge q^{\frac{1}{10}}$ by the following argument. We have
\begin{align}
\label{discard}  \sum_{\substack{q^{\frac{1}{10}} \le n\le q^{\frac{51}{50}}\\ (n,q)=1}}   \frac{1}{n^\half \phi(\mathrm{ord}(n))}  \ll q^{-\frac{1}{20}}    \sum_{\substack{1 \le n\le q^{\frac{51}{50}}\\ (n,q)=1}}    \frac{1}{\phi(\mathrm{ord}(n))} \ll q^{-\frac{1}{20}}   q^{\frac{1}{50}} \sum_{ 1\le n \le q-1}    \frac{1}{\phi(\mathrm{ord}(n))}, 
\end{align}
 where the last bound follows because $\ord(n)$ only depends on the residue class of $n$ modulo $q$, and there are $O(q^{\frac{1}{50}})$ translates of the interval $[1,q]$ in the interval $1 \le n\le q^{\frac{51}{50}}$. Now since $\ord(n)$ must divide $q-1$, we have that \eqref{discard} is bounded by 
 \[
 q^{-\frac{3}{100}}   \sum_{ k|(q-1)}   \frac{1}{\phi(k)} \sum_{\substack{1\le n \le q-1\\ \mathrm{ord}(n)=k}} 1 \ll q^{-\frac{3}{100}}   \sum_{ k|(q-1)}   \frac{1}{\phi(k)} \phi(k) \ll q^{-\frac{1}{100}},
 \]
using the fact that there are $\phi(k)$ elements in $[1,q-1]$ with order $k$ (see \cite[Theorem 110]{hardy-wright}) and that the number of divisors of $q-1$ is  bounded by $q^\varepsilon$ for any $\varepsilon>0$ (see \cite[Theorem 317]{hardy-wright}). Thus Theorem \ref{main} is reduced to proving
 \begin{prop} \label{prop}
 \begin{align}
 \label{thesum} \sum_{\substack{2\le n< q^{\frac{1}{10}}}}   \frac{1}{n^\half \phi(\mathrm{ord}(n))} \ll (\log q)^{-\delta}.
 \end{align}
 \end{prop}
\noindent Of course, there is nothing special about the exponents $\frac{51}{50}$ and $\frac{1}{10}$ in Lemma \ref{afe} and Proposition \ref{prop}. By adjusting the proofs, we could replace them with $1+\varepsilon$ and $4\varepsilon$ for any arbitrarily small fixed $\varepsilon>0$, but this would not matter. 
 
 
We study the sum \eqref{thesum} in dyadic intervals. Thus define
 \begin{align*}
S(N):=\sum_{\substack{N\le n< 2N}} \, \frac{1}{n^\half\phi(\ord(n))},
\end{align*}
and note that 
\[
\sum_{\substack{2\le n < q^{\frac{1}{10}}}}   \frac{1}{n^\half \phi(\mathrm{ord}(n))} \le \sum_{j=1}^{ \lfloor \frac{1}{10} \log_2 q \rfloor} S(2^j). 
\]
We will treat the sums $S(N)$ in two ranges of $N$ to prove Proposition \ref{prop}. There are $O((\log q)^{\frac{1}{10}})$ dyadic intervals with $2\le N \le \exp((\log q)^{\frac{1}{10}})$, so for such $N$ we need to show that $S(N)\ll (\log q)^{-\frac{1}{10} -\delta}$ for some $\delta>0$. This is accomplished in Lemma \ref{range1}. There are $O(\log q)$ dyadic intervals with $\exp((\log q)^{\frac{1}{10}})\le N < q^\frac{1}{10}$, so for such $N$ we need to show that $S(N)\ll (\log q)^{-1 -\delta}$ for some $\delta>0$. This is done in Lemma \ref{range2}.


\section{Proof}

First, we need to develop some control over the distribution of $\ord(n)$. For $k|(q-1)$, let 
\begin{align}
\label{alphadef} \alpha_N(k):=|\{ N\le n < 2N: \ord(n)=k \}|.
\end{align} 
The trivial bound for this is 
\begin{align}
\label{trivbou} \alpha_N(k)\le \min\{N, \phi(k)\},
\end{align}
as there are $\phi(k)$ elements in $[1,q-1]$ with order $k$, but we need a better estimate.

\begin{lem} \label{alpha-lemma} For integers $2 \le N< q^{\frac{1}{10}}$ and $k|(q-1)$, we have
\[
\alpha_N(k) \ll \frac{k \log N}{\log q}.
\]
\end{lem}
\proof
If there is no element of order $k$ in the interval $[N,2N)$, then we are done. So suppose $a\in[N,2N)$ has order $k$. Then $a^k\equiv 1 \bmod q$, which implies $a^k = 1+ qj$ for some $j\ge 1$ because $a\ge N\ge 2$ rules out the possibility that $j=0$. Thus $a^k\ge q$, which gives the lower bound
\begin{align}
\label{lowerboundk} k\ge\ \frac{\log q}{\log a} \gg \frac{\log q}{\log N}.
\end{align}

Clearly, $\alpha_N(k) \le \beta_N(k)$, where 
\[
\beta_N(k)=|\{ N\le n< 2N: \ord(n)|k \}|.
\]
The set of elements in $(\mathbb{Z}/q\mathbb{Z})^\times$ whose order divides $k$ are congruent mod $q$ to $\{ a, a^2, a^3, \ldots, a^k \}$.  So the set of elements in $[1,q-1]$ whose order divides $k$ is 
\[
\mathcal{A}=\{ a \, \% \, q, a^2 \, \% \, q, a^3 \, \% \, q , \ldots, a^k \, \% \, q \},
\] 
where $a^j \, \% \, q$ denotes the remainder (a value between 1 and $q-1$) when $a^j$ is divided by $q$.

Suppose that $a^r \, \% \, q \in [N,2N)$ for some $1\le r\le k$. Let $1 \le j \le \frac{\log q}{\log 2N}-1$. Then since $a\in[N,2N)$ and $N\ge 2$, we have 
\begin{align}
\label{inequal} 2N\le N^2 \le (a^r \, \% \, q)a^j < (2N)^{1+j} \le q.
\end{align}
Since $a^{r+j} \, \% \, q$ is congruent to $(a^r \, \% \, q)a^j$ mod $q$, and both quantities lie in $[1,q-1]$, we have that $a^{r+j} \, \% \, q = (a^r \, \% \, q)a^j$. Then by \eqref{inequal}, we have $a^{r+j} \, \% \, q \notin [N,2N)$. Thus considering the elements $a^r \, \% \, q$ sequentially, for $1\le r\le k$, whenever an element lies in $[N,2N)$, the next $\lfloor \frac{\log q}{\log 2N}-1 \rfloor \gg \frac{\log q}{\log N}$ or more elements (unless the list is exhausted) do not lie in $[N,2N)$. Thus 
\[
\beta_N(k) \ll \frac{k}{ \frac{\log q}{\log N}} + 1 \ll \frac{k}{ \frac{\log q}{\log N}},
\]
using \eqref{lowerboundk}.
\endproof


We are now ready to prove our bounds for $S(N)$. The argument is different according to the size of $N$. For smaller values of $N$ we have

\begin{lem} \label{range1} For integers $ 2 \le N\le \exp((\log q)^\frac{1}{10})$, we have $S(N) \ll (\log q)^{-\frac{2}{5}}.$
\end{lem}
\proof
Recall the definition \eqref{alphadef}. We have
\begin{align}
\label{sum_k} S(N) =\sum_{k|(q-1)} \frac{1}{\phi(k)}\sum_{\substack{N\le n< 2N\\ \ord(n)=k} }  \frac{1}{n^\half} \ll \frac{1}{N^\half}\sum_{k|(q-1)} \frac{ \alpha_N(k)}{\phi(k)}\ll \frac{\log\log q}{N^\half}\sum_{k|(q-1)} \frac{ \alpha_N(k)}{k},
\end{align}
using \eqref{phi-lower}. By Lemma \ref{alpha-lemma}, we have that
\begin{align}
\label{alpha-bound} \alpha_N(k)\le C \frac{k \log N }{\log q}
\end{align}
for some constant $C>100$, say.
Furthermore, as there are $N$ integers in the interval $[N,2N)$, and each has a unique order, we have
\begin{align}
\label{alpha-total}  \sum_{k|(q-1)} \alpha_N(k) = N.
\end{align}
Now to find an upper bound for the sum
\[
\mathcal{S}:= \sum\limits_{k|(q-1)} \frac{ \alpha_N(k)}{k},
\]
we employ a greedy strategy. By \eqref{alpha-total}, the problem is to distribute the total $N$ among the values $\alpha_N(k)$, subject to the pointwise bounds \eqref{alpha-bound}, so as to maximize $\mathcal{S}$. Since the coefficients $\frac{1}{k}$ strictly decrease with $k$, the sum $\mathcal{S}$ is maximized by taking 
\[
\alpha_N(k)= C \frac{k \log N}{\log q}
\]
for all small $k$ first, filling greedily from $k=1$ upwards. For each divisor $k$ of $q-1$, let
\[
\iota(k):=\#\{d\mid (q-1):d\le k\}
\]
denote its index among the positive divisors of $q-1$ listed in increasing order. Suppose that $\alpha_N(k)= C \frac{k \log N}{\log q}$ for $\iota(k)\le K$, $\alpha_N(k)< C \frac{k \log N}{\log q}$ for $\iota(k)=K+1$, and $\alpha_N(k)=0$ for $\iota(k)\ge K+2$. Then we have 
\[
C \sum_{\substack{k\le K}} \frac{k \log N}{\log q} \le C \sum_{\substack{k|(q-1)\\ \iota(k)\le K}} \frac{k \log N}{\log q}  \le  \sum_{\substack{k|(q-1)}} \alpha_N(k) = N.
 \]
Since the leftmost sum equals $C \frac{\log N}{\log q} \frac{K(K+1)}{2}$, this implies that $K\le  \frac12 N^\half (\frac{\log q}{\log N})^\half $. Thus by \eqref{sum_k}, we have
\begin{align}
\label{sum_k2} S(N) &\ll \frac{\log\log q}{N^\half} \sum_{\substack{k|(q-1)\\ \iota(k) \le  N^\half  (\frac{\log q}{\log N})^\half }} \frac{ \frac{k\log N}{\log q}}{k}\ll  \frac{(\log\log q)(\log N)}{N^\half \log q} \sum_{\substack{ k\le  N^\half  (\frac{\log q}{\log N})^\half }}  1\ll \frac{(\log N)^\half \log\log q}{(\log q)^\half}.
\end{align}
Inserting the upper bound $N\le \exp((\log q)^\frac{1}{10})$, we get 
\[
S(N)\ll  \frac{(\log q)^\frac{1}{20} \log\log q}{(\log q)^\half}\ll (\log q)^{-\frac25}.
\]
\endproof


For larger values of $N$, we have

\begin{lem}\label{range2} For integers $ \exp((\log q)^\frac{1}{10}) \le N < q^{\frac{1}{10}}$, we have $S(N) \ll (\log q)^{-2}$.
\proof
Let 
\begin{align*}
\mathcal{K}=   N^\half  (\log q)^3
\end{align*}
and 
\begin{align*}
 P=  \exp(  (\log q)^\frac{1}{20}).
\end{align*}
Note that in the range $ \exp((\log q)^\frac{1}{10}) \le N < q^{\frac{1}{10}}$ under consideration, we have
\begin{align}
\label{note} P \le (N^\half)^{2(\log q)^{-\frac{1}{20}}} \le \mathcal{K}^{2(\log q)^{-\frac{1}{20}}}.
\end{align}

Recall from \eqref{sum_k} that
\begin{align*}
S(N)\ll \frac{1}{N^\half}\sum_{\substack{k|(q-1)}} \, \frac{\alpha_N(k)}{\phi(k)}.
\end{align*}
Now splitting this sum according the size of $k$, we have
\begin{align*}
S(N) \ll \frac{1}{N^\half}\sum_{\substack{k|(q-1)\\ k\le \mathcal{K}}} \, \frac{\alpha_N(k)}{\phi(k)} 
+\frac{1}{N^\half}\sum_{\substack{k|(q-1)\\ k>\mathcal{K}}} \, \frac{\alpha_N(k)}{\phi(k)}.
\end{align*}
For the sum restricted to $k\le \mathcal{K}$, we use \eqref{trivbou}. For the sum restricted to $k> \mathcal{K}$, we use \eqref{phi-lower} and \eqref{alpha-total}. Thus we get
\begin{align}
\label{divided} S(N) \ll 
\frac{d(q-1;\mathcal{K})}{N^\half}
+\frac{N^\half \log\log q}{\mathcal{K}}.
\end{align}
Clearly, the second term in \eqref{divided} satisfies the required bound $O((\log q)^{-2})$. We now focus on the first term.

Every divisor of $q-1$ bounded by $\mathcal{K}$ either has at least one prime factor larger than $P$, or it is an integer less than $\mathcal{K}$ with all prime factors bounded by $P$. Thus
\begin{align}
\label{est-0} d(q-1;\mathcal{K})\le  d_P (q-1;\mathcal{K})+ \Psi(\mathcal{K}; P ).
\end{align}
Since $P^{(\log q)^{\frac{19}{20}}}>q-1$, we have that $q-1$ has at most $(\log q)^{\frac{19}{20}}$ prime factors $p$ larger than $P$. Thus
\begin{align}
\label{est-1} d_P (q-1;\mathcal{K})\ll  \sum_{\substack{p|(q-1)\\ p\ge P }} \, \sum_{\substack{d\le \mathcal{K} \\ p|d }}1 \ll  \sum_{\substack{p|(q-1)\\ p\ge P }} \Big\lfloor \frac{\mathcal{K}}{p}\Big\rfloor \ll (\log q)^{\frac{19}{20}} \frac{\mathcal{K}}{P} \ll \mathcal{K} \exp( -(\log q)^{\frac{1}{30}} )\ll N^{\frac12}(\log q)^{-2}.
\end{align}
By Lemma \ref{smooth} and \eqref{note}, we have
\begin{align}
\label{est-2} \Psi(\mathcal{K}; P ) \ll \Psi(\mathcal{K};  \mathcal{K}^{2(\log q)^{-\frac{1}{20}}} )\ll \mathcal{K} \exp( -\tfrac14 (\log q)^{\frac{1}{20}} ) \ll N^{\frac12}(\log q)^{-2}.
\end{align}
By \eqref{est-0}, \eqref{est-1}, \eqref{est-2}, we get that the first term in \eqref{divided} satisfies the required bound $O((\log q)^{-2})$.
\endproof
\end{lem}

By the discussion at the end of Section \ref{sec}, the bounds given in Lemmas \ref{range1} and \ref{range2} are sufficient to complete the proof of Proposition \ref{prop}, and hence Theorem \ref{main}.

\

{\bf Acknowledgments.} I thank Agniva Dasgupta for some valuable comments.

\bibliographystyle{amsalpha}
\bibliography{first}

\end{document}